\documentclass[10pt]{article}

\usepackage[
  paperwidth=6in,
  paperheight=8.625in,
  textwidth=4.5in,
  textheight=7.125in,
  centering
]{geometry}

\usepackage{amsmath,amssymb,amsthm}
\usepackage{enumitem}
\usepackage{authblk}

\newtheorem{theorem}{Theorem}

\newcommand{\set}[1]{\{#1\}}

\begin{document}

\begin{center}
{\bfseries\Large An explicit family of 30 blocks meeting every $6$-set of $[60]$ in at least two points\par}
\vspace{1.0ex}
{\normalsize Paulo Henrique Cunha Gomes\par}
{\normalsize School of Technology (FT) – University of Campinas (Unicamp),Limeira,13484-332,São Paulo,Brazil\par}
{\normalsize \texttt{e-mail:p072049@dac.unicamp.br}\par} 
\end{center}

\vspace{1.25ex}

\noindent\textsc{abstract}\par
\noindent
We exhibit an explicit family $\mathcal{B}$ of $30$ subsets (``blocks'') of size $6$ of
$[60]=\{1,2,\dots,60\}$ with the following property: for every $6$-subset $S\subset[60]$,
there exists a block $B\in\mathcal{B}$ such that $|S\cap B|\ge 2$.
The construction is fully explicit and the proof is purely combinatorial.

\vspace{0.8ex}
\par

\section{Introduction}
Let $[60]=\{1,2,\dots,60\}$. We construct an explicit family $\mathcal{B}$ of $30$ blocks
(each a $6$-subset of $[60]$) such that for every $6$-subset $S\subset[60]$ there exists
$B\in\mathcal{B}$ with $|S\cap B|\ge 2$. For general background on combinatorial designs,
see \cite{CD07}.

\section{Construction}
Any partition $[60]$ into ten \textbf{disjoint} base blocks of size $6$,
not necessarily one of the partitions $G_i$ will present:

\[
G_i=\set{6(i-1)+1,\,6(i-1)+2,\,\dots,\,6i}\qquad (i=1,2,\dots,10).
\]
Group these ten blocks into five fixed pairs:
\[
(G_1,G_2),\ (G_3,G_4),\ (G_5,G_6),\ (G_7,G_8),\ (G_9,G_{10}).
\]
Inside each base block $G_i=\{a_1,a_2,a_3,a_4,a_5,a_6\}$, split it into two triples
\[
G_i^{(1)}=\set{a_1,a_2,a_3},
\qquad
G_i^{(2)}=\set{a_4,a_5,a_6}.
\]
For each paired couple $(G_{2m-1},G_{2m})$ (with $m\in\{1,2,3,4,5\}$), form the four
recombined blocks
\[
G_{2m-1}^{(u)}\cup G_{2m}^{(v)}
\qquad\text{for }u,v\in\{1,2\}.
\]
Finally, define $\mathcal{B}$ as the family consisting of the $10$ base blocks
$G_1,\dots,G_{10}$ together with the $20$ recombined blocks above. Hence
$|\mathcal{B}|=10+20=30$.

\section{Main result}
\begin{theorem}\label{thm:main}
Let $\mathcal{B}$ be the family of $30$ blocks constructed above. For every $S\subset[60]$ with
$|S|=6$, there exists $B\in\mathcal{B}$ such that $|S\cap B|\ge 2$.
Equivalently: every $6$-subset $S$ contains a pair $\{x,y\}\subset S$ that is contained in at least
one block $B\in\mathcal{B}$.
\end{theorem}

\begin{proof}
Fix $S\subset[60]$ with $|S|=6$.

\smallskip
\noindent\textbf{Case 1.} If $S$ contains two elements from the same base block $G_i$, then
$|S\cap G_i|\ge 2$ and we are done since $G_i\in\mathcal{B}$.

\smallskip
\noindent\textbf{Case 2.} Otherwise, $S$ meets each base block in at most one element. Since
$|S|=6$, the set $S$ meets exactly six distinct base blocks among $G_1,\dots,G_{10}$.

The ten base blocks are partitioned into five pairs
\[
(G_1,G_2),\ (G_3,G_4),\ (G_5,G_6),\ (G_7,G_8),\ (G_9,G_{10}).
\]
By the pigeonhole principle, among the six base blocks hit by $S$, at least two belong to the same
pair, say $(G_{2m-1},G_{2m})$. Thus $S$ contains one element
$x\in G_{2m-1}$ and one element $y\in G_{2m}$.

By construction, $x\in G_{2m-1}^{(u)}$ for some $u\in\{1,2\}$ and
$y\in G_{2m}^{(v)}$ for some $v\in\{1,2\}$. Therefore the recombined block
\[
B:=G_{2m-1}^{(u)}\cup G_{2m}^{(v)}
\]
belongs to $\mathcal{B}$ and contains both $x$ and $y$, hence $|S\cap B|\ge 2$.
\end{proof}

\section{An explicit example (the 30 blocks)}
Below we list an explicit instance of the construction. The reader may choose any $6$ numbers
(i.e., any $6$-subset $S\subset[60]$) and verify that at least one block below contains at least two
elements of $S$, as guaranteed by Theorem~\ref{thm:main}.

\subsection*{Base blocks (10)}
\begin{enumerate}[label=\textbf{\arabic*.},leftmargin=*,itemsep=0.15em]
\item $\{1,2,3,4,5,6\}$
\item $\{7,8,9,10,11,12\}$
\item $\{13,14,15,16,17,18\}$
\item $\{19,20,21,22,23,24\}$
\item $\{25,26,27,28,29,30\}$
\item $\{31,32,33,34,35,36\}$
\item $\{37,38,39,40,41,42\}$
\item $\{43,44,45,46,47,48\}$
\item $\{49,50,51,52,53,54\}$
\item $\{55,56,57,58,59,60\}$
\end{enumerate}

\subsection*{Recombined blocks (20)}
\begin{enumerate}[label=\textbf{\arabic*.},leftmargin=*,itemsep=0.15em,resume]
\item $\{1,2,3,7,8,9\}$
\item $\{1,2,3,10,11,12\}$
\item $\{4,5,6,7,8,9\}$
\item $\{4,5,6,10,11,12\}$
\item $\{13,14,15,19,20,21\}$
\item $\{13,14,15,22,23,24\}$
\item $\{16,17,18,19,20,21\}$
\item $\{16,17,18,22,23,24\}$
\item $\{25,26,27,31,32,33\}$
\item $\{25,26,27,34,35,36\}$
\item $\{28,29,30,31,32,33\}$
\item $\{28,29,30,34,35,36\}$
\item $\{37,38,39,43,44,45\}$
\item $\{37,38,39,46,47,48\}$
\item $\{40,41,42,43,44,45\}$
\item $\{40,41,42,46,47,48\}$
\item $\{49,50,51,55,56,57\}$
\item $\{49,50,51,58,59,60\}$
\item $\{52,53,54,55,56,57\}$
\item $\{52,53,54,58,59,60\}$
\end{enumerate}

\section{Final remark}
This construction can be viewed as a covering-type property in the Johnson space $J(60,6)$:
every $6$-subset intersects at least one of the $30$ blocks in at least two elements.

\end{document}